\documentclass[preprint,review,12pt]{elsarticle}
\usepackage{amssymb}
\usepackage[leqno]{amsmath}
\usepackage{amsfonts}
\usepackage{amsthm}
\usepackage{cases}
\usepackage{graphicx}
\usepackage{subfigure}
\usepackage{epstopdf}
\usepackage{epsfig}
\usepackage{color}

\usepackage{mathrsfs}
\usepackage{mathabx}

 \usepackage{diagbox}

\numberwithin{equation}{section}

\makeatletter
\def\fnum@figure{Fig.\thefigure}
\makeatother

\newtheorem*{assumption*}{Assumption}

\newtheorem{assumption}{Assumption}

\newtheorem{lemma}{Lemma}
\newtheorem{remark}{Remark}
\newtheorem{theorem}{Theorem}

\let \mr=\mathrm
\let \b=\boldsymbol

\begin{document}

\begin{frontmatter}
\title{Convergence and supercloseness in a balanced norm  of finite
element methods on Bakhvalov-type  meshes for
reaction-diffusion problems\tnoteref{funding}}
\tnotetext[funding]{This research is partially supported by National Natural Science Foundation of China (11771257,11601251)}

\author[label1] {Jin Zhang\corref{cor1}}
\author[label2] {Xiaowei Liu \fnref{cor2}}
\cortext[cor1] {Corresponding author: jinzhangalex@hotmail.com }
\fntext[cor2] {Email: xwliuvivi@hotmail.com }
\address[label1]{School of Mathematics and Statistics, Shandong Normal University,
Jinan 250014, China}
\address[label2]{School of Mathematics and Statistics, Qilu University of Technology (Shandong Academy of Sciences), Jinan 250353, China}

\begin{abstract}
In convergence analysis of finite element methods for singularly perturbed reaction--diffusion problems,
balanced norms have been successfully introduced to replace standard energy norms so that layers can be captured. In this article, we focus on the convergence analysis in a balanced norm  on Bakhvalov-type rectangular meshes. In order to achieve our goal, a novel interpolation operator, which  consists of a local weighted $L^2$ projection operator and the Lagrange interpolation operator, is introduced for a convergence analysis of optimal order in the balanced norm. The analysis  also depends on the stabilities of the $L^2$ projection and the characteristics of Bakhvalov-type meshes. Furthermore,  we obtain a supercloseness result  in the balanced norm, which appears in the literature for the first time. This result depends on another novel interpolant, which consists of the local weighted  $L^2$ projection operator, a  vertices-edges-element operator and some corrections on the boundary. 
\end{abstract}


\end{frontmatter}

%
%
 
\section{Introduction}
 In this article we consider the  singularly perturbed reaction-diffusion equation 
\begin{equation}\label{eq:reaction diffusion equation}
\begin{split}
-\varepsilon^2\Delta u+bu=&f\quad \text{in $\Omega:=(0,1)^2$},\\
u=&0\quad \text{on $\partial\Omega$},
\end{split}
\end{equation}
where $\varepsilon$ is a   positive parameter.
Assume that $b$ and $f$ are sufficiently smooth and   
$$
b(x,y)\ge 2\beta^2>0 \quad \forall (x,y)\in\bar{\Omega},
$$
with a positive constant $\beta$. Under these conditions on the data in problem \eqref{eq:reaction diffusion equation}, there exists a unique solution  in  $H^1_0(\Omega) \cap H^2(\Omega)$ for all $f\in L^2(\Omega)$. 
The solution to problem \eqref{eq:reaction diffusion equation} typically exhibits boundary  layers of width $\mathcal{O}(\varepsilon |\ln \varepsilon|)$ along all of $\partial\Omega$ in the 
singularly perturbed case $0<\varepsilon\ll 1$ of interest.

For singularly perturbed problems, it is popular to introduce layer-adapted meshes  \cite{Linb:2010-Layer,Roo1Sty2Tob3:2008-Robust,Stynes:2005-Steady}   to fully resolve layers. Then uniform convergence with respect to singular perturbation parameter can be achieved for standard numerical methods. There are two   kinds of layer-adapted meshes  widely used in the literature, which are Bakhvalov-type mesh  (B-mesh) and Shishkin-type mesh (S-mesh) (see \cite{Linb:2010-Layer}). There are a lot of research results on convergence theories of finite element methods on S-meshes; see   \cite{Roo1Sty2Tob3:2008-Robust,Franz:2008-Singularly,Dur1Lom2Pri3:2013-Supercloseness,Zhan1Liu2:2016-Analysis-CL,Zhan1Liu2:2017-Supercloseness-EL,Zhan1Styn2:2017-Supercloseness,Liu1Sty2Zha3:2018-Supercloseness-CIP-EL}
and references therein. 

Although  B-meshes usually have better performances than  S-meshes, there are very few articles on uniform convergence    of finite element methods on the former.
One main reason is that B-meshes have specific transition points between the fine and coarse parts, which are independent of mesh parameter. In the meantime, they bring great difficulties to convergence analysis. For example,  Lagrange interpolant does not work well for B-meshes. Recently, Zhang and Liu \cite{Zhan1Liu2:2020-Optimal,Zhan1Liu2:2021-Supercloseness} proposed an variant of Lagrange interpolant for finite element methods on B-meshes in the case of  convection-diffusion equations and  succeeded to obtain a uniform convergence of optimal order.  

 For reaction-diffusion problems, they also proved  optimal order of uniform convergence in the natural energy norm in \cite{Zhan1Liu2:2020-Convergence}.  However,  energy norm  is not strong enough to capture layers as the singular perturbation parameter tends to zero. 
Thus balanced norms, which are stronger than standard energy norms and characterize layers in an more appropriate way, were introduced in \cite{Lin1Styn2:2012-balanced} for a mixed finite element method and \cite{Roos1Scho2:2015-Convergence} for a finite element method.  
The authors  \cite{Roos1Scho2:2015-Convergence} introduced an $L^2$-projection to   obtain    desired estimations 
 by  $L^{\infty}$-stability of the $L^2$-projection
   on Shishkin mesh. To improve estimations in \cite{Roos1Scho2:2015-Convergence}, the authors  \cite{Fran1Roos2:2019-Error}  introduced a new interpolation, which consists a local weighted $L^2$ projection defined on the uniform part of S-meshes.  Unfortunately,  unlike S-meshes, there is little development on convergence theories in balanced norms on  B-meshes.

In this manuscript we analyze  convergence theories in the balanced norm introduced in \cite{Roos1Scho2:2015-Convergence} 
 for $k$th ($k\ge 1$) order finite element method on Bakhvalov-type rectangular   meshes. For this purpose, we propose a novel interpolant  according to the structures of B-meshes and layer functions. This interpolant consists of a local weighted $L^2$ projection defined on a proper mesh subdomain and the Lagrange interpolant. To prove the convergence of optimal order in the balanced norm, we must take into account the scales of the meshes and different stabilities of the $L^2$ projection.   
The optimal order  convergence  is also supported by our numerical experiments.  Furthermore, we propose another novel interpolant operator, which consists of the local weighted $L^2$ projection operator, a  vertices-edges-element operator and some corrections on the boundary. By careful derivations, we obtain a supercloseness result, which appears in the literature for the first time. Here ``supercloseness''  means that the convergence order  for the error between some interpolation of the solution $u$ and the numerical solution  $u^{N}$ in some norm is greater than the order for $u-u^N$ in the same norm.

The rest of the paper is organized as follows. In  Section~\ref{sec:FE+mesh} we present a priori information of the solution to~\eqref{eq:reaction diffusion equation},  then introduce Bakhvalov-type meshes,  finite element methods and some preliminary results. In Section~\ref{sec:inter+convergence} we give a new interpolant   and prove   uniform convergence of optimal order in the balanced norm. Supercloseness result is given in Section~\ref{sec:supercloseness} by means of another novel interpolant. In Section~\ref{sec:numerical}, numerical results illustrate our theoretical results.

%

Let $D\subset \Omega$. In this article, we will write $(\cdot, \cdot)_{D}$ for the inner product in $L^2(D)$,  $\Vert \cdot \Vert_D$, $\Vert \cdot \Vert_{\infty,D}$, $\Vert \cdot \Vert_{1,D}$ and $\vert \cdot \vert_{1,D}$ for the standard norms in $L^2(D)$, $L^{\infty}(D)$, $L^{1}(D)$ and the standard seminorm in $H^1(D)$, respectively. 
 If $D=\Omega$, the subscript will be omitted from the above norm designations.  
 Throughout the paper, all constants $C$ and $C_i$ are  independent of $\varepsilon$ and~$N$; the constants $C$ are generic while subscripted constants $C_i$ are fixed.

%
%
\section{Finite element method on Bakhvalov-type mesh}\label{sec:FE+mesh}

%
%
%
\subsection{Regularity results}
To construct layer-adapted meshes and analyze uniform convergence, we need a priori information of the solution $u$  to \eqref{eq:reaction diffusion equation}, such as pointwise estimations of the derivatives of the solution, the locations and widths of layers.

For this aim, we give the following assumption on  the solution $u$ to \eqref{eq:reaction diffusion equation}
according to \cite{Han1Kell2:1990-Differentiability,Liu1Mad2Sty3etc:2009-two-scale,Cla1Gra2Ori3:2005-parameter}.
\begin{assumption}\label{assumption-regularity}
The solution $u$ of \eqref{eq:reaction diffusion equation} can be decomposed as
\begin{equation}\label{eq:(2.1a)}
u=v_0+\sum_{i=1}^4w_i+\sum_{i=1}^4z_i \quad \forall \b{x}\in\bar{\Omega},
\end{equation}
where $v_0$ is the regular part, each $w_i$ is a boundary layer function  and each $z_i$ is a corner layer  function. For   $k\in\mathbb{N}$ and $k\ge 1$, there exists a constant $C$ such that 
\begin{align}
&\left|\partial^{m}_{x}\partial^{n}_{y}v_0(x,y) \right|\le C(1+\varepsilon^{(k+1)-m-n})\quad \text{for $0\le m+n\le k+3$},\label{eq:(2.1b)}\\
&\left|\partial^{m}_{x}\partial^{n}_{y}w_1(x,y) \right|\le C(1+\varepsilon^{(k+1)-m})\varepsilon^{-n}e^{-\beta y/\varepsilon}\quad \text{for $0\le m+n\le k+2$},\label{bound-w-1}\\
&\left|\partial^{m}_{x}\partial^{n}_{y}z_1(x,y) \right|\le C\varepsilon^{-(m+n)}e^{-\beta (x+y)/\varepsilon} \quad \text{for $0\le m+n\le k+2$},\label{bound-z-1}
\end{align}
and similarly for the remaining terms. Here denote $\dfrac{\partial^{m+n} v }{ \partial x^m \partial y^n }$ by $\partial^{m}_{x}\partial^{n}_{y}v$  .
\end{assumption}
In the following analysis, we will  denote   $ \sum\limits_{i=1}^4w_i+\sum\limits_{i=1}^4z_i$ by $w$.
 
\subsection{Bakhvalov-type meshes}\label{sec:layer-mesh}
Two  Bakhvalov-type meshes will be discussed. Let $N\in \mathbb{N}$  be divisible by 4. The first  Bakhvalov-type mesh is   introduced in \cite{Roos-2006-Error} and defined by
\begin{equation}\label{eq:Bakhvalov mesh-Roos}
x_i=y_i=\psi(i/N)=
\left\{
\begin{split}
& -\frac{\sigma \varepsilon}{\beta} \ln ( 1-4(1-\varepsilon)i/N ) \quad &&\text{for $i=0,\ldots,N/4$},\\
&d_1 (i/N-1/4)+d_2(i/N-3/4)\quad &&\text{for $i=N/4,\ldots,3N/4$},\\
&1+\frac{\sigma \varepsilon}{\beta} \ln ( 1-4(1-\varepsilon)(1-i/N) ) \quad &&\text{for $i=3N/4,\ldots,N$},
\end{split}
\right.
\end{equation}
where  $\sigma$ will be defined later and $d_1$, $d_2$ are used to ensure the continuity of $\psi(t)$ at $t=1/4$ and $t=3/4$.  The second Bakhvalov-type mesh is introduced in  \cite{Kopteva:1999-the,Kopt1Save2:2011-Pointwise}  and its  mesh generating function is
\begin{equation}\label{eq:Bakhvalov mesh-Kopteva}
\varphi(t)=
\left\{
\begin{split}
&-\frac{\sigma \varepsilon}{\beta}  \ln(1-4t)  \quad &&\text{for $t\in [0,\vartheta]$},\\
&d_3 (t-\vartheta)+d_4(t-1+\vartheta)\quad &&\text{for $t\in (\vartheta,1-\vartheta)$},\\
&1+\frac{\sigma \varepsilon}{\beta}  \ln(1-4(1-t))  \quad &&\text{for $t\in [1-\vartheta,1]$}
\end{split}
\right.
\end{equation}
where $\sigma$  will be specified later, $\vartheta=1/4-C_1\varepsilon$ with some positive constant $C_1$ independent of $\varepsilon$ and $N$,  $d_3$ and $d_4$ are chosen so that $\varphi(t)$  is continuous at $t=\vartheta$ and $t=1-\vartheta$. 
The original Bakhvalov mesh \cite{Bakhvalov:1969-Towards} can be recovered from \eqref{eq:Bakhvalov mesh-Kopteva}
by setting $\vartheta=1/4-\mathcal{C}(\varepsilon)\varepsilon$ with $
0<C_2\le \mathcal{C}(\varepsilon) \le C_3$.

\begin{assumption}\label{assumption-varepsilon-N}
Assume  that $\varepsilon\le \min\{ \frac{\beta}{4\sigma},1\} N^{-1}$ in our analysis. In practice it is not a restriction. 
\end{assumption}

  For technical reasons, we also assume   
\begin{equation}\label{eq:CC1}
\frac{\sigma }{ 4e \beta  }\le C_1\le   \frac{1}{2}\max\{ \frac{\sigma}{\beta},\frac{1}{4}\}.  
\end{equation}
 Assume  $N\ge \max\{ 8, 2  \ln \frac{\sigma }{\beta}\}$.   Under Assumption \ref{assumption-varepsilon-N} and \eqref{eq:CC1}, we have $1/4-N^{-1}\le  \vartheta< 1/4$ and  $x_{N/4}\le 1/4$
  for    meshes   \eqref{eq:Bakhvalov mesh-Roos}  and \eqref{eq:Bakhvalov mesh-Kopteva}.    
 The location of  $\vartheta$ and conditions imposed on $N$ and $C_1$ will simplify our later analysis without changing the essential difficulties in our analysis.

The mesh points are $x_i=y_i=\psi(i/N)$ or $x_i=y_i=\varphi(i/N)$ for $i=0,1,\ldots,N$.
By drawing lines parallel to the axis through   mesh points $\{(x_i,y_j)\}$,  we obtain a Bakhvalov-type rectangular 
mesh with equidistant cells in the coarse region 
$\Omega_0=(x_{N/4},x_{3N/4})^2$ and anisotropic cells in
the layer region $\Omega\setminus \Omega_0$.  The triangulation is denoted by $\mathcal{T}^{N}$. Denote  by $\tau_{i,j}$ for the element $[x_{i},x_{i+1}]\times [y_j,y_{j+1}]$ and by $\tau$ for a generic  rectangular  element, which dimensions are written as $h_{x,\tau}$ and $h_{y,\tau}$. Define
$h_i:=x_{i+1}-x_i=y_{i+1}-y_i$ for $0\le i\le N-1$.

In the following lemma, we collect some important properties possessed by Bakhvalov-type meshes, which are important for convergence analysis. The reader is referred to \cite{Zhan1Liu2:2020-Optimal} for the detailed proof.
\begin{lemma}\label{lem:h-N-4-1}
Let Assumption \ref{assumption-varepsilon-N} hold true.  On  Bakhvalov-type  mesh  \eqref{eq:Bakhvalov mesh-Roos} or \eqref{eq:Bakhvalov mesh-Kopteva},  one has
\begin{align}
&h_{0}\le h_{1}\le \ldots \le h_{N/4-2},\label{eq:mesh-1}\\
&\frac{\sigma}{4\beta} \varepsilon \le h_{N/4-2}\le \frac{\sigma}{ \beta} \varepsilon,\label{eq:mesh-4}\\
&C\varepsilon\le h_{N/4-1} \le CN^{-1},\label{eq:mesh-5}\\
 &C_4N^{-1}\le h_{i}\le C_5 N^{-1}\quad  N/4 \le i \le N/2,\label{eq:mesh-6} \\
&x_{N/4-1}=x_{3N/4+1}\ge C\sigma \varepsilon \ln N,\quad
x_{N/4}=x_{3N/4}\ge C\sigma \varepsilon \ln (1/\varepsilon).
\end{align}
Note $h_i=h_{N-1-i}$ for $i=0,1,\ldots,N/2$.

Let $i^*=j^*=N/4-2$. Then one has
\begin{equation}\label{eq:layer-function-4} 
h_{i}^{ \mu} e^{- \beta x_i/\varepsilon}\le C\varepsilon^{\mu} N^{-\mu}\quad \text{for $0\le i\le i^*$  and $0\le \mu\le \sigma$}.
\end{equation}
Similar bounds hold for the variable $y$.
\end{lemma}


\subsection{Finite element method}
Now we present the finite element method for problem \eqref{eq:reaction diffusion equation}. First, 
 the weak form of
 problem \eqref{eq:reaction diffusion equation} is written as  
\begin{equation}\label{eq:weak formulation}
\left\{
\begin{array}{lr}
\text{find $u\in V$ such that for all $v\in V$}\\
a(u,v):=\varepsilon^2 (\nabla u,\nabla v)+(bu,v)=(f,v),
\end{array}
\right.
\end{equation}
with $V:=H^{1}_{0}(\Omega)$. The natural energy norm derived from $a(\cdot,\cdot)$ is    
$$
\Vert v \Vert_{\varepsilon}:=\left(\varepsilon^2 \vert v  \vert^2_{1}+  \Vert v\Vert^2\right)^{1/2}.
$$
The bilinear form $a(\cdot,\cdot)$ is coercive with respect to this energy norm, i.e.,
\begin{equation}\label{eq:energy-norm}
a(v,v)\ge \min\{2\beta^2,1\}\Vert v \Vert^2_{\varepsilon} \quad \forall v\in V.
\end{equation}
From the Lax-Milgram lemma, the weak formulation \eqref{eq:weak formulation} has a unique solution.

 Let $\mathcal{Q}_{k}(\tau)$   denote $k$th rectangular finite element. We introduce the finite element spaces
$$
V^N:=\{v\in H^1(\Omega):\;v|_{\tau}\in \mathcal{Q}_{k}(\tau)\quad \forall \tau\in\mathcal{T}^N \}\text{ and } V^N_0:=V^N\cap H^1_0(\Omega).
$$
Clearly, $V^N,V^N_0\subset V$. When we replace the infinite dimensional space $V$ with the finite dimensional space $V^N_0$, we get the $k$th order finite element method  
\begin{equation}\label{eq:FEM}
\left\{
\begin{array}{lr}
\text{find $u^N\in V^N_0$ such that for all $v^N_0\in V^N$}\\
a(u^N,v^N) =(f,v^N).
\end{array}
\right.
\end{equation}
Also, it is easy to very the  coercivity
\begin{equation}\label{eq:coercivity}
a(v^N,v^N)\ge C\Vert v^N \Vert^2_{\varepsilon}\quad\forall v^N\in V^N_0,
\end{equation}
and the Galerkin orthogonality
\begin{equation}\label{eq:orthogonality}
a(u-u^N,v^N)=0\quad\forall v^N\in V^N_0,
\end{equation}
where \eqref{eq:weak formulation} and \eqref{eq:FEM} have been used. Furthermore, we introduce
the balanced norm  in
 \cite{Roos1Scho2:2015-Convergence}, which is defined  by
 \begin{equation}
 \Vert v \Vert_b :=\left(\varepsilon\vert v \vert_1^2+\Vert v \Vert^2 \right)^{1/2}.
 \end{equation}
Clearly, the balanced norm $\Vert \cdot \Vert_b$ is stronger than the energy norm $\Vert \cdot \Vert_{\varepsilon}$ in the case of $0<\varepsilon\ll 1$.  Furthermore, the former is better suited to  capture of layers. For example,
for a typical layer 
function $e^{-x/\varepsilon}$, $\Vert e^{-x/\varepsilon} 
\Vert_{\varepsilon}$ and $\Vert e^{-x/\varepsilon} \Vert_{b}$ 
are of order $\mathcal{O}(\varepsilon^{1/2})$ and of order $
\mathcal{O}(1)$, respectively. These orders imply that the balanced norm $\Vert \cdot \Vert_b$ is more appropriate to capture layers than the energy norm $\Vert \cdot \Vert_{\varepsilon}$ when $0<\varepsilon\ll 1$. The reader is also referred to 
\cite{Roos1Scho2:2015-Convergence,Con1Fra2Lud3etc:2018-Finite} for discussions on these two norms.

\section{Uniform convergence}\label{sec:inter+convergence}
For convergence analysis in the balanced norm, we will present an interpolation operator, which components will be introduced  at first.

Set 
$\Omega_0^*:=(x_{N/4-1},x_{3N/4+1})^2$ and $\Omega^{**}_0:=(x_{N/4-2},x_{3N/4+2})^2$.  
Introduce a weighted  $L^2$-projection  $\pi$  as follows: for $s\in L^2(\Omega_0^*)$, find $\pi s\in W^N $ such that
$$
(b(s-\pi s),v^N)_{\Omega_0^*}=0\quad \forall v^N\in W^N, 
$$
where $W^N=:\{v|_{\Omega_0^*}:\;v\in V^N \}$.
 Of course, one has the   $L^2$-stability  
\begin{equation}\label{eq:L2-stability}
\Vert \pi v\Vert_{\Omega_0^*}\le C \Vert v \Vert_{\Omega_0^*}.
\end{equation} 
Denote by $\mathcal{I}$ the  Langrange interpolation operator from $C^0(\bar{\Omega})$ to $V^N$.
Furthermore,  
define $\chi\in V^N_0$ by
$$
\chi(s_l, t_m)=
\left\{
\begin{aligned}
&1\quad &&(s_l, t_m)\in \partial \Omega_0^*,\\
&0\quad && \text{otherwise},
\end{aligned}
\right.
$$
where   $\{s_l, t_m\}$ are the interpolation points of the Lagrange interpolation. 

Recall $w=\sum\limits_{i=1}^4w_i+\sum\limits_{i=1}^4z_i$. Then the interpolation used in convergence analysis is  defined by 
\begin{equation}\label{eq:projector-1}
P_cu=P_1v_0+P_2w,
\end{equation}
where  
\begin{equation}\label{eq:projector-2}
P_1v_0 :=
\left\{
\begin{aligned}
&\pi v_0 \quad &&\text{ in  $\Omega_0^*$},\\
&\mathcal{I}[(1-\chi )v_0+\chi \pi v_0]\quad&&\text{ in  $\Omega\setminus \Omega^{*}_0$}
\end{aligned}
\right.
\end{equation}
and
\begin{equation}\label{eq:projector-3}
P_2w :=
\left\{
\begin{aligned}
&0\quad &&\text{ in  $\Omega^*_0$},\\
&\mathcal{I}[(1-\chi  )w]\quad&&\text{ in  $\Omega\setminus \Omega^*_0$}.
\end{aligned}
\right.
\end{equation}
Clearly, $P_cu\in V^N_0$.

The following lemma provides some pointwise bounds  for errors between the Lagrange interpolant and the $L^2$ projection.
\begin{lemma}\label{lem:main}
For $v_0$ introduced in Assumption \ref{assumption-regularity},  one has
\begin{align*}
&\Vert \mathcal{I}v_0-\pi v_0 \Vert_{\infty,\partial\Omega^*_0}+\Vert \mathcal{I}v_0-\pi v_0 \Vert_{\infty, \Omega^*_0 }+\Vert  v_0-\pi v_0 \Vert_{\infty, \Omega^*_0 }\le CN^{-(k+1)},\\
&\Vert \nabla (\mathcal{I}v_0-\pi v_0)\Vert_{\Omega_0} \le CN^{-k},\quad
\Vert \nabla (\mathcal{I}v_0-\pi v_0)\Vert_{\Omega^*_0\setminus\Omega_0}\le C\varepsilon^{-1/2} N^{-(k+1)}.
\end{align*}
\end{lemma}
\begin{proof}
Standard interpolation theory and Lemma  \ref{lem:h-N-4-1} imply 
$\Vert v_0-\mathcal{I}v_0   \Vert_{\infty, \Omega^*_0  }
+\Vert v_0- \mathcal{I}v_0  \Vert_{  \Omega^*_0  }
\le  CN^{-(k+1)}$.
From the $L^{\infty}$-stability of the $L^2$-projection $\pi$ \cite[Theorem 1]{Crou1Tome2:1987-stability}, one has
\begin{align*}
&\Vert \mathcal{I}v_0-\pi v_0 \Vert_{\infty, \Omega^*_0 }
=\Vert \pi(\mathcal{I}v_0- v_0) \Vert_{\infty, \Omega^*_0 }
\le  C \Vert \mathcal{I}v_0-  v_0 \Vert_{\infty, \Omega^*_0  }    
\le  CN^{-(k+1)},\\
& \Vert \mathcal{I}v_0-\pi v_0 \Vert_{\infty,\partial\Omega^*_0}\le
\Vert \mathcal{I}v_0-\pi v_0 \Vert_{\infty, \Omega^*_0 } 
\le  CN^{-(k+1)},\\
& \Vert  v_0-\pi v_0 \Vert_{\infty, \Omega^*_0}\le
\Vert  v_0-\mathcal{I}v_0 \Vert_{\infty, \Omega^*_0}+
\Vert \mathcal{I}v_0-\pi v_0 \Vert_{\infty, \Omega^*_0 } 
\le  CN^{-(k+1)}.
\end{align*} 

H\"{o}lder inequalities, inverse inequalities, the $L^{\infty}$-stability of the $L^2$-projection \cite[Theorem 1]{Crou1Tome2:1987-stability} and Lemma \ref{lem:h-N-4-1} yield
\begin{align*}
\Vert \nabla (\mathcal{I}v_0-\pi v_0)\Vert_{\Omega^*_0\setminus\Omega_0}
\le &
\mr{meas}(\Omega^*_0\setminus\Omega_0)^{1/2}\Vert \nabla (\mathcal{I}v_0-\pi v_0)\Vert_{\infty,\Omega^*_0\setminus\Omega_0}\\
\le &Ch^{-1/2}_{N/4-1}  \Vert  \mathcal{I}v_0-\pi v_0 \Vert_{\infty,\Omega^*_0\setminus\Omega_0}\\
=&Ch^{-1/2}_{N/4-1}  \Vert  \pi(\mathcal{I}v_0- v_0) \Vert_{\infty,\Omega^*_0 }\\
\le &Ch^{-1/2}_{N/4-1}  \Vert  \mathcal{I}v_0-  v_0 \Vert_{\infty,\Omega^*_0}\\
\le &C\varepsilon^{-1/2} N^{-(k+1)}.
\end{align*}

Inverse inequalities and the $L^2$-stability of the $L^2$-projection \eqref{eq:L2-stability} yield
\begin{align*}
&\Vert \nabla (\mathcal{I}v_0-\pi v_0)\Vert_{\Omega_0} 
\le CN \Vert  \mathcal{I}v_0-\pi v_0 \Vert_{\Omega_0}\le CN \Vert  \mathcal{I}v_0-\pi v_0 \Vert_{\Omega_0^*}\\
\le &CN \Vert  \mathcal{I}v_0-  v_0 \Vert_{\Omega^*_0}  
\le  CN^{-k}.
\end{align*}
\end{proof}

\begin{remark}
Define 
$$
V_h=\{v\in C[x_{N/4-1},x_{3N/4+1}]; \; v|_{(s_j,s_{j+1})}\in P_k,\;j=0,\ldots,N/2+1 \}
$$ 
with $s_j=x_{N/4-1+j}$ for $j=0,\ldots,N/2+2$. 
In the same way as the proof of \cite[Theorem 1]{Crou1Tome2:1987-stability}, we could easily prove the $L^{\infty}$-stability of the $L^2$-projection $\pi_h:\;L^2(x_{N/4-1},x_{3N/4+1})\to V_h$, that is 
$$
\Vert  \pi_h v \Vert_{\infty,(x_{N/4-1},x_{3N/4+1})}
\le C\Vert   v \Vert_{\infty,(x_{N/4-1},x_{3N/4+1})}
\quad
\forall v\in L^{\infty}(x_{N/4-1},x_{3N/4+1}).
$$
Furthermore, from tensor product we could easily obtain
$$
\Vert  \pi  v \Vert_{\infty,\Omega_0^*}
\le
C \Vert     v \Vert_{\infty,\Omega_0^*}
\quad
\forall v\in L^{\infty}(\Omega_0^*).
$$

\end{remark}

The error $u-u^N$ is split   as follows:
$$
u-u^N=(u-P_cu)+(P_cu-u^N)=:\eta+\xi.
$$
From the coercivity \eqref{eq:coercivity} and the Galerkin orthogonality \eqref{eq:orthogonality} we have
\begin{equation}\label{eq:whole idea}
 C \Vert \xi \Vert^2_{\varepsilon}\le   a(P_c u-u,\xi) 
\le  \varepsilon^{1/2} \Vert \eta \Vert_b \Vert  \xi \Vert_{\varepsilon}+|( b \eta, \xi )|.
\end{equation}

In the following analysis,  we give the estimations on each term in the right-hand side of \eqref{eq:whole idea}. 
\begin{lemma}\label{lem:1}
Let $\sigma\ge k+1$. Under  Assumptions \ref{assumption-regularity} and \ref{assumption-varepsilon-N} we have
 $$
|(b\eta,\xi)|\le C \varepsilon^{1/2} N^{-(k+1)} \ln^{1/2} N \Vert \xi \Vert.
$$
\end{lemma}
\begin{proof}
Our arguments are based on the following splitting
\begin{align*}
(b\eta,\xi)=&(b(v_0-P_1v_0),\xi)+(b(w-P_2w),\xi)\\
=&(b(v_0-\pi v_0),\xi)_{\Omega_0^*}+(b(v_0-\mathcal{I}v_0),\xi)_{\Omega\setminus\Omega_0^*}
+(b\;\mathcal{I}[\chi  (v_0-\pi v_0) ],\xi)_{\Omega^{**}_0\setminus\Omega^*_0}\\
&+(bw,\xi)_{\Omega^*_0}+(b(w-\mathcal{I}w),\xi)_{\Omega\setminus\Omega^*_0}
+(b\;\mathcal{I}(\chi   w),\xi)_{\Omega^{**}_0\setminus\Omega^*_0}\\
=:&\mr{I}+\mr{II}+\mr{III}+\mr{IV}+\mr{V+\mr{VI}}.
\end{align*}

From the definition of $\pi$, we obtain
\begin{equation}\label{eq:I}
\mr{I}=0.
\end{equation}

From   \cite[Lemma 4]{Zhan1Liu2:2020-Convergence},  one has
\begin{equation}\label{eq:II}
\begin{aligned}
|\mr{II}|+|\mr{V}|\le & C \left( \Vert v_0-\mathcal{I}v_0 \Vert_{\infty,\Omega\setminus\Omega^*_0}  
+  \Vert w-\mathcal{I}w \Vert_{\infty,\Omega\setminus\Omega^*_0} \right) \Vert \xi \Vert_{1,\Omega\setminus\Omega^*_0  } \\
\le & C   N^{-(k+1)} \mr{meas}^{1/2}(\Omega\setminus\Omega^*_0)\; \Vert \xi \Vert_{\Omega\setminus\Omega^*_0} \\
\le & 
C\varepsilon^{1/2} N^{-(k+1)} \ln^{1/2} N\Vert \xi \Vert.
\end{aligned}
\end{equation}

Lemmas \ref{lem:h-N-4-1} and \ref{lem:main}, \eqref{bound-w-1}  and \eqref{bound-z-1} yield
\begin{equation}\label{eq:III+V}
\begin{aligned}
|\mr{III}|+|\mr{VI}|\le &C\left( \Vert \mathcal{I}v_0-\pi v_0 \Vert_{\infty,\partial\Omega^*_0} 
+\Vert \mathcal{I}w \Vert_{\infty,\partial\Omega^*_0}  
 \right)\Vert \xi \Vert_{1,\Omega^{**}_0\setminus\Omega^*_0}\\
 \le &
 C\left( N^{-(k+1)}  +N^{-\sigma} \right)h^{1/2}_{N/4-2} \Vert \xi \Vert_{\Omega^{**}_0\setminus\Omega^*_0}\\
 \le &
 C \varepsilon^{1/2 } N^{-(k+1)}  \Vert \xi \Vert.
 \end{aligned}
\end{equation}

From \eqref{bound-w-1}  and \eqref{bound-z-1}, we get
\begin{equation}\label{eq:IV}
|\mr{IV}|\le C \Vert w \Vert_{\Omega^*_0} \Vert \xi \Vert_{\Omega^*_0}
\le C \varepsilon^{1/2} N^{-\sigma} \Vert \xi \Vert.
\end{equation}

Collecting \eqref{eq:I}--\eqref{eq:IV}, we are done.
\end{proof}

\begin{lemma}\label{lem:2}
Let $\sigma\ge k+1$. Under  Assumptions \ref{assumption-regularity} and \ref{assumption-varepsilon-N} we have
$$
\Vert  \eta\Vert_b\le C N^{-k}.
$$
\end{lemma}

\begin{proof}
From \eqref{eq:projector-1}, \eqref{eq:projector-2} and \eqref{eq:projector-3} one has
$$
\begin{aligned}
\Vert \eta \Vert_b\le  &\Vert  v_0-\pi v_0  \Vert_{b,\Omega^*_0}+\Vert  v_0-\mathcal{I}v_0  \Vert_{b,\Omega\setminus\Omega^*_0}+\Vert \mathcal{I}(\chi   (v_0-\pi v_0) )  \Vert_{b,\Omega^{**}_0\setminus \Omega^*_0}\\
&+
\Vert  w-\mathcal{I}w  \Vert_{b,\Omega\setminus\Omega^*_0}+\Vert  \mathcal{I}(\chi  w)  \Vert_{b,\Omega^{**}_0\setminus \Omega^*_0 }\\
=:&S_1+S_2+S_3+S_4+S_5.
\end{aligned}
$$

To analyze $S_1$,  we need the following bounds 
\begin{align*}
\Vert \nabla (v_0-\pi v_0) \Vert_{\Omega_0}
\le & \Vert \nabla (v_0-\mathcal{I}v_0)\Vert_{\Omega_0}+\Vert \nabla (\mathcal{I}v_0-\pi v_0)\Vert_{\Omega_0} 
\le   CN^{-k}, \\
 \Vert \nabla (v_0-\pi v_0) \Vert_{\Omega^*_0\setminus\Omega_0}
\le & \Vert \nabla (v_0-\mathcal{I}v_0)\Vert_{\Omega^*_0\setminus\Omega_0}+\Vert \nabla (\mathcal{I}v_0-\pi v_0)\Vert_{\Omega^*_0\setminus\Omega_0}\\
\le & CN^{-k}+C\varepsilon^{-1/2} N^{-(k+1)},\\
\Vert  v_0-\pi v_0  \Vert_{\Omega^*_0}
\le &
\Vert  v_0-\mathcal{I}v_0   \Vert_{\Omega^*_0}
+\Vert   \mathcal{I}v_0-\pi v_0  \Vert_{\Omega^*_0}\\
\le &   CN^{-(k+1)}+C\Vert   \mathcal{I}v_0-\pi v_0  \Vert_{\infty,\Omega^*_0}\le  CN^{-(k+1)},
\end{align*}
which could be derived from standard interpolation theories  and Lemma \ref{lem:main}.
Then  we obtain
\begin{equation}\label{eq:S-1}
\begin{aligned}
|S_1|\le &\varepsilon^{1/2}\Vert \nabla (v_0-\pi v_0)\Vert_{\Omega^*_0}+ \Vert  v_0-\pi v_0 \Vert_{\Omega^*_0}  
\le   C \varepsilon^{ 1/2} N^{-k} +C N^{-(k+1)}.
\end{aligned}
\end{equation}

Similar to \cite[Lemma 4]{Zhan1Liu2:2020-Convergence}, we have
$$
\Vert v_0-\mathcal{I}v_0 \Vert_{\infty,\Omega\setminus\Omega^*_0}
+\Vert w-\mathcal{I}w \Vert_{\infty,\Omega\setminus\Omega^*_0}
\le C N^{-(k+1)}.
$$
Imitating the proof of Lemma 5 in \cite{Zhan1Liu2:2020-Convergence} and replacing \cite[(3.15)]{Zhan1Liu2:2020-Convergence} by $|w_1|_{1,[0,1]\times [y_{N/4-1},1]}\le \varepsilon^{-1/2} N^{-(k+1)}$, we obtain
$$
\vert w_1-\mathcal{I}w_1 \vert_{1,\Omega\setminus\Omega^*_0}\le C \varepsilon^{-1/2}N^{-k},
$$
and similarly have
$$
\vert w-\mathcal{I}w \vert_{1,\Omega\setminus\Omega^*_0}\le C \varepsilon^{-1/2}N^{-k}.
$$
Thus we obtain
\begin{equation}\label{eq:S-2+4}
S_2+S_4\le C N^{-k}.
\end{equation}

Inverse inequalities, Lemmas \ref{lem:h-N-4-1} and \ref{lem:main} yield 
\begin{equation}
\begin{aligned}
 S_3  \le &\varepsilon^{1/2} \Vert \nabla \left( \mathcal{I}(\chi  (v_0-\pi v_0) ) \right) \Vert_{\Omega^{**}_0\setminus\Omega^*_0}+\Vert  \mathcal{I}(\chi   (v_0-\pi v_0) )  \Vert_{\Omega^{**}_0\setminus\Omega^*_0}\\
\le  & C (\varepsilon^{1/2}h^{-1}_{N/4-2} +1)\Vert  \mathcal{I}(\chi  (v_0-\pi v_0) ) \Vert_{\Omega^{**}_0\setminus\Omega^*_0} \\
\le & C (\varepsilon^{1/2}h^{-1}_{N/4-2} +1)h^{1/2}_{N/4-2}\Vert \mathcal{I}v_0-\pi v_0 \Vert_{  \infty, \partial \Omega^*_0   }
 \\
\le & C N^{-(k+1)}.
 \end{aligned}
\end{equation}
Similarly, one has
\begin{equation}\label{eq:S5}
S_5 \le C N^{-\sigma}. 
\end{equation}

Collecting \eqref{eq:S-1}--\eqref{eq:S5}, we are done.
\end{proof}

Similar to \cite[Theorem 2.6]{Fran1Roos2:2019-Error}, we obtain the following theorem from Lemmas \ref{lem:1} and \ref{lem:2}. 
\begin{theorem}\label{the:main-result}
Let $\sigma\ge k+1$. Let Assumptions \ref{assumption-regularity} and \ref{assumption-varepsilon-N} hold. Then for the exact solution $u$ to 
\eqref{eq:reaction diffusion equation} and the numerical solution $u^N$ to \eqref{eq:FEM} on  Bakhvalov-type rectangular mesh \eqref{eq:Bakhvalov mesh-Roos} or \eqref{eq:Bakhvalov mesh-Kopteva}, one has
$$
\Vert u-u^N \Vert_b\le C  N^{-k}.
$$
\end{theorem}
\section{Supercloseness}\label{sec:supercloseness}
In order to derive the supercloseness result in the balanced norm, we need another novel interpolant, which will be described in the following.

Instead of Langrange interpolant operator used in the previous section, we introduce an vertices-edges-element interpolation operator $\mathcal{A}:\;C^0(\bar{\Omega})\to V^N$ (see \cite{Lin1Yan2Zho3:1991-rectangle,Styn1Tobi2:2008-Using}).  This interpolant is used for superconvergence analysis of the diffusion term. 
First we define the interpolant operator on the reference element $\hat{\tau}=(-1,1)^2$, whose vertices and edges are denoted by $\hat{a}_i$ and $\hat{e}_i$ respectively for $i=1,\ldots,4$.  Let $\hat{v}(\cdot,\cdot)\in C(\overline{\hat{\tau} })$.
 The operator $\hat{\mathcal{A}}:\; C(\overline{\hat{\tau} })\to \mathcal{Q}_k(\hat{\tau})$ is determined by $(k+1)^2$   continuous linear functionals $\hat{F}:\;C (\overline{\hat{\tau}})\to \mathbb{R}$,
 which are defined by 
\begin{align*}
&\hat{v}\to \hat{v}(\hat{a}_i)\quad \text{i=1,\ldots,4},\\
&\hat{v}\to \int_{\hat{e}_i}\hat{v}q\mr{d}s\quad\forall q\in\mathcal{P}_{k-2}(\hat{e}_i) \;\;\text{i=1,\ldots,4},\\
&\hat{v}\to \int_{\hat{\tau} }\hat{v}q\mr{d}x\mr{d}y\quad\forall q\in\mathcal{Q}_{k-2}(\hat{\tau}).
\end{align*}
 From \cite[Lemma 3]{Styn1Tobi2:2008-Using}, the operator $\hat{ \mathcal{A} }$ is uniquely determined.  Then using the affine  transformation
to map from $\hat{\tau}$ to an arbitrary $\tau\in \mathcal{T}^N$, one obtains the corresponding interpolation operator $\mathcal{A}_{\tau}:\;C(\bar{\tau})\to\mathcal{Q}_k(\tau)$. At last a continuous global interpolation operator $\mathcal{A} :\;C(\bar{\Omega})\to V^N$ is defined by setting
$$
(\mathcal{A} v)|_{\tau}:=\mathcal{A}_{\tau}(v|_{\tau})
\quad \forall \tau\in \mathcal{T}^N.
$$
Besides, we denote by $F$  degree   of freedom (DoF) of $V^N$, which originates  from the linear functional  $\hat{F}$.

Recall $i^*=j^*=N/4-2$ and   $\Omega_{w_1}:=[0,1]\times [y_0,y_{j^*+1}]$.
%
 The operator $S_1:\;C^0(\overline{\Omega})\to V^N$ is defined by
\begin{equation}\label{eq:projector-2-S}
S_1w_1 := 
\left\{
\begin{aligned}
&0\quad &&\text{ in  $\Omega\setminus \Omega_{w_1}$},\\
&\mathcal{A}w_1-\mathcal{B}_1w_1
 \quad&&\text{ in  $\Omega_{w_1}$},
\end{aligned}
\right.
\end{equation}
where 
$\mathcal{B}_1w_1$ satisfies
$$
F( \mathcal{B}_1w_1 )=
\left\{
\begin{aligned}
&F(w_1)\quad && \text{if $F$ is the   DoF of $V^N$ attached to
$[0,1]\times \{y=y_{j^*+1} \} $},\\
&0\quad && \text{otherwise}.
\end{aligned}
\right.
$$ 
If $F$ is the  DoF of $V^N$ attached to
$[0,1]\times \{y=y_{j^*+1} \} $, then it must be one of the following forms
\begin{align*}
& v\to v(x_i,y_{j^*+1})\quad \text{for $i=0,\ldots,N$},\\
&v\to \frac{j+1}{h_{i}^{j+1}}\int_{x_{i}}^{x_{i+1}}v(s,y_{j^*+1})(s-x_i)^j\mr{d}s \;\;\text{for $i=0,\ldots,N-1$ and $j=0,\ldots,k-2$}.
\end{align*}
In fact, $\mathcal{B}_1w_1$ is introduced for the continuity of $S_1w_1$ at $[0,1]\times \{ y=y_{j^*+1} \}$. Set $\Omega_{z_1}:=[x_0,x_{i^*+1}]\times [y_0,y_{j^*+1}]$. 
  The operator $T_1:\;C^0(\overline{\Omega})\to V^N$ is defined by
\begin{equation}\label{eq:projector-2-S}
T_1z_1:= 
\left\{
\begin{aligned}
&0\quad &&\text{ in  $\Omega\setminus \Omega_{z_1}$},\\
&\mathcal{A}z_1-\mathcal{C}_1z_1
\quad&&\text{ in  $ \Omega_{z_1}$},
\end{aligned}
\right.
\end{equation}
where the operator $\mathcal{C}_1$ is defined in a similar way to $\mathcal{B}_1$ except that the degrees of freedom for $\mathcal{C}_1$ are attached to $[0,x_{i^*+1}]\times \{ y=y_{j^*+1} \} \cup \{ x=x_{i^*+1} \} \times [0,y_{j^*+1}]$. 
  The operators $S_i$ and $T_i$ for $i=2,3,4$ could be defined similarly. Also we give a boundary correction  $\mathcal{C}(S_1 w_1)\in V^N$ for $S_1 w_1$, which is defined by
$$
F(\mathcal{C}(S_1 w_1) )=
\left\{
\begin{aligned}
&F(w_1)\quad && \text{if $F$ is the DoF of $V^N$ attached to
$\Gamma_{w_1}$},\\
&0\quad && \text{otherwise},
\end{aligned}
\right.
$$ 
where 
$ 
\Gamma_{w_1}:=\overline{ \partial\Omega\setminus \partial\Omega_{w_1}}.
$ 
With the help of this correction, we have
$$
\left(S_1w_1+\mathcal{C}(S_1w_1) \right)|_{\partial\Omega}= \mathcal{A}w_1|_{\partial\Omega}.
$$
By the same token, we could define the corrections $\mathcal{C}(S_iw_i)$ and $\mathcal{C}(T_iz_i)$ for $S_iw_i$ and $T_iz_i$ for $i=1,2,3,4$, respectively.

Introduce the discrete function $\mathcal{D}v_0\in V^N_0 $, which is defined by
$$
F(\mathcal{D}v_0 )=
\left\{
\begin{aligned}
&F(\pi v_0-v_0)\quad && \text{if $F$ is the DoF of $V^N$ attached to
$ \partial \Omega_0^*$},\\
&0\quad && \text{otherwise},
\end{aligned}
\right.
$$
and
\begin{equation}\label{eq:projector-2-S}
\mathcal{E}v_0 :=
\left\{
\begin{aligned}
&\pi v_0 \quad &&\text{ in  $\Omega_0^*$},\\
&
\mathcal{A}v_0+\mathcal{D}v_0
\quad&&\text{ in  $\Omega\setminus \Omega^{*}_0$.}
\end{aligned}
\right.
\end{equation}
Note that $(\mathcal{D}v_0)|_{\partial \Omega^{*}_0}=\left( \pi v_0-\mathcal{A}v_0 \right)|_{\partial \Omega^{*}_0}$. The definition of $\mathcal{D}v_0$ ensures the continuity  of $\mathcal{E}v_0$ on $\bar{\Omega}$.

Now we are in a position to propose the interpolation used for  our supercloseness analysis, which is defined by 
\begin{equation}\label{eq:projector-1-S}
P_su=\mathcal{E}v_0+\sum_{i=1}^4 S_iw_i+\sum_{i=1}^4 T_iz_i
+\mathcal{C}(w),
\end{equation}
where $\mathcal{C}(w)=\sum\limits_{i=1}^4 \mathcal{C}(S_iw_i)+\sum\limits_{i=1}^4 \mathcal{C}(T_iz_i)$. Clearly $P_su\in V^N_0$.

The following lemma could be found in \cite[Lemma 4]{Styn1Tobi2:2008-Using}, which is important for analysis of the diffusion part. 
\begin{lemma}\label{lem:superconvergence}
Let $\tau\in \mathcal{T}^N$. Let $v\in H^{k+2}(\tau)$ and 
$\mathcal{A}_{\tau}v\in\mathcal{Q}_{k}(\tau)$ be its vertices-edges-element interpolant. Then for each $v^N\in \mathcal{Q}_{k}(\tau)$ we have
$$
\left|\int_{\tau} (\mathcal{A}_{\tau}v-v)_x v^N_x\mr{d}x\mr{d}y \right|\le C h_{y,\tau}^{k+1}
\left\Vert \frac{\partial^{k+2}v }{\partial x\partial y^{k+1}} \right\Vert_{\tau} \Vert v^N_x \Vert_{\tau}
$$
and
$$
\left|\int_{\tau} (\mathcal{A}_{\tau}v-v)_y v^N_y\mr{d}x\mr{d}y \right|\le C h_{x,\tau}^{k+1}
\left\Vert \frac{\partial^{k+2}v }{\partial x^{k+1}\partial y} \right\Vert_{\tau} \Vert v^N_y \Vert_{\tau}.
$$
\end{lemma}

For the vertices-edges-element interpolation, we have the following interpolation errors.
\begin{lemma}\label{lem:VEE-interpolation-error}
Let $\tau\in \mathcal{T}^N$. Let $v\in H^{k+1}(\Omega)$. Then there exists a constant $C$  such that the vertices–edges–element interpolation $\mathcal{A}v$ satisfies
$$
\Vert v- \mathcal{A}v \Vert_{\tau}\le
C\sum_{i+j=k+1}h_{x,\tau}^i h_{y,\tau}^j
\left\Vert \frac{\partial^{k+1} v }
{\partial x^i \partial y^j} \right\Vert_{\tau}. 
$$
\end{lemma}
\begin{proof}
See \cite[Lemma 7]{Styn1Tobi2:2008-Using}.
\end{proof}

Set 
$$
e(D)=\{\tau\in\mathcal{T}^N:\;\bar{\tau}\cap (\overline{ \Omega\setminus D })\ne \varnothing,\;
\tau \in  D \}.
$$
 The following $L^{\infty}$-stabilities will be used later.
\begin{lemma}\label{lem:B-stability}
Let $v\in C^0(\bar{\Omega})$. There exists a constant $C$ independent of $v$ such that
\begin{align*}
&\Vert   \mathcal{B}_1v  \Vert_{\infty,e(\Omega_{w_1})}\le  C\Vert v  \Vert_{\infty,e(\Omega_{w_1})},\\
&\Vert ( \mathcal{B}_1v )_x \Vert_{\infty,e(\Omega_{w_1})}\le  C \Vert v_x \Vert_{\infty,e(\Omega_{w_1})},\\
&\Vert  \mathcal{C}_1z_1 \Vert_{\infty,e(\Omega_{z_1})} 
\le C \Vert   z_1 \Vert_{\infty,e(\Omega_{z_1})}.
\end{align*}
\end{lemma}
\begin{proof}
For $\tau_{i,j^*}$ with $i=0,\ldots,N-1$, one has
$$
\begin{aligned}
\mathcal{B}_1v|_{\tau_{i,j^*} }=&v(x_i,y_{j^*+1})\varphi_{i}(x,y)
+v(x_{i+1},y_{j^*+1})\varphi_{i+1}(x,y)\\
&+\sum_{m=0}^{k-2}\frac{m+1}{h_{i,x}^{m+1}}\int_{x_{i}}^{x_{i+1}}v(s,y_{j^*+1})(s-x_i)^m\mr{d}s\; \varphi_{i,m}(x,y). 
\end{aligned}
$$
The inequality $\Vert   \mathcal{B}_1v  \Vert_{\infty,\tau_{i,j^*}}\le  C\Vert v  \Vert_{\infty,\tau_{i,j^*}}$ follows. In a similar way, we could prove $\Vert  \mathcal{C}_1z_1 \Vert_{\infty,e(\Omega_{z_1})})
\le C \Vert   z_1 \Vert_{\infty,e(\Omega_{z_1})}$. 
Note
$$
v(s,y_{j^*+1})=v(x,y_{j^*+1})+
\int_x^s v_t(t,y_{j^*+1})\mr{d}t\quad \text{for $s,x\in [x_i,x_{i+1}]$},
$$ 
and  
$$
\left(\varphi_{i}(x,y)+\sum_{m=0}^{k-2}\varphi_{i,m}(x,y)+\varphi_{i+1}(x,y) \right)_x\equiv 0 \quad \text{for $(x,y)\in \tau_{i,j^*}$}.
$$ 
Then we have
$$
\begin{aligned}
(\mathcal{B}_1v)_x|_{\tau_{i,j^*} }=&
\int_x^{x_i} v_t(t,y_{j^*+1})\mr{d}t \; (\varphi_{i}(x,y))_x
+\int_x^{x_{i+1} } v_t(t,y_{j^*+1})\mr{d}t  \;(\varphi_{i+1}(x,y))_x\\
&+\sum_{m=0}^{k-2}\frac{m+1}{h_{i,x}^{m+1}}\int_{x_{i}}^{x_{i+1}} \left( \int_x^{s} v_t(t,y_{j^*+1})\mr{d}t \right) (s-x_i)^m\mr{d}s\; (\varphi_{i,m}(x,y))_x,
\end{aligned}
$$
and
$$
\Vert (\mathcal{B}_1v)_x \Vert_{\infty,\tau_{i,j^*}}
\le CM\Vert v_x  \Vert_{\infty,\tau_{i,j^*}} 
\le C \Vert v_x  \Vert_{\infty,\tau_{i,j^*}}
$$
where from scaling arguments one has
$$
M=h_{i}
\max\limits_{
\genfrac{}{}{0em}{}{l=i,i+1}
{m=0,\ldots,k-2}
}\{ \Vert (\varphi_{l}(x,y))_x \Vert_{\infty,\tau_{i,j^*}},  \Vert (\varphi_{i,m}(x,y))_x \Vert_{\infty,\tau_{i,j^*}} \}\le C.
$$
Consider $e(\Omega_{w_1})=\cup_{i=0}^{N-1}\tau_{i,j^*}$ and we are done.
\end{proof}

The error is split   as follows:
$$
u-u^N=(u-P_su)+(P_su-u^N)=:\tilde{\eta}+ \tilde{\xi}.
$$
From the coercivity \eqref{eq:coercivity}, the Galerkin orthogonality \eqref{eq:orthogonality} and \eqref{eq:projector-1-S},  one has
\begin{equation}\label{eq:idea-supercloseness}
\begin{aligned}
 C \Vert \tilde{\xi} \Vert^2_{\varepsilon}\le &  a(P_s u-u,\tilde{\xi}) 
= \varepsilon^{2}( \nabla \tilde{\eta}, \nabla  \tilde{\xi})+  ( b \tilde{\eta}, \tilde{\xi} )\\
=:&\varepsilon^{2}\mathscr{S}_1+\mathscr{S}_2,
\end{aligned}
\end{equation}
where
\begin{align*}
\mathscr{S}_1=&( \nabla (v_0-\mathcal{E}v_0), \nabla  \tilde{\xi}) +\sum_{i=1}^4( \nabla (w_i-S_iw_i), \nabla  \tilde{\xi})+  \sum_{i=1}^4( \nabla (z_i-T_iz_i), \nabla  \tilde{\xi}),\\ 
\mathscr{S}_2=&(b(v_0-\mathcal{E}v_0),\tilde{\xi})+\sum_{i=1}^4( b \left( (w_i-S_iw_i)+(z_i-T_iz_i) \right), \tilde{\xi} )+a( \mathcal{C}(w),\tilde{\xi}).
\end{align*}

The terms in the right-hand side of \eqref{eq:idea-supercloseness} will be analyzed in the following two lemmas.
\begin{lemma}\label{lem:diffusion}
Let Assumptions \ref{assumption-regularity} and \ref{assumption-varepsilon-N} hold. Let $\sigma\ge k+3/2$.  Then one has
$$
|\mathscr{S}_1|\le C\varepsilon^{-3/2}(\varepsilon^{1/2}N^{-k}+N^{-(k+1)}) \Vert \tilde{\xi} \Vert_{\varepsilon}.
$$
\end{lemma}
\begin{proof}
From Assumption  \ref{assumption-regularity}, we just analyze $(\nabla (v_0-\mathcal{E} v_0), \nabla \tilde{\xi})$, $( \nabla (w_1-S_1w_1),  \nabla \tilde{\xi})$  and $( (z_1-T_1z_1)_x,  \tilde{\xi}_x)$. The remaining terms can be treated in a similar manner. We split these  terms as follows:
\begin{equation}\label{eq:diffusion}
(\nabla (v_0-\mathcal{E} v_0), \nabla \tilde{\xi})+( \nabla (w_1-S_1w_1),  \nabla \tilde{\xi})+( (z_1-T_1z_1)_x,  \tilde{\xi}_x)=\sum_{i=1}^4\mathcal{S}_i
\end{equation}
where
\begin{align*}
\mathcal{S}_1
=& (\nabla (v_0-\mathcal{A} v_0), \nabla \tilde{\xi})_{\Omega\setminus\Omega^*_0}+ ( \nabla  (w_1-\mathcal{A}w_1),  \nabla \tilde{\xi})_{ \Omega_{w_1} }+( (z_1-\mathcal{A}z_1)_x,  \tilde{\xi}_x)_{ \Omega_{z_1} }, \\
\mathcal{S}_2
=&(( \mathcal{B}_1w_1)_x,\tilde{\xi}_x)_{ e(\Omega_{w_1})  },\\
\mathcal{S}_3
=&
 (( \mathcal{B}_1w_1)_y,\tilde{\xi}_y)_{ e(\Omega_{w_1})  }+(( \mathcal{C}_1z_1 )_x,\tilde{\xi}_x)_{ e(\Omega_{z_1})  }, \\
\mathcal{S}_4
=& 
( \nabla  w_1 ,  \nabla \tilde{\xi})_{  \Omega\setminus\Omega_{w_1} }+( (z_1)_x,  \tilde{\xi}_x)_{  \Omega\setminus\Omega_{z_1} },\\
\mathcal{S}_5
=& (\nabla (v_0-\pi v_0), \nabla \tilde{\xi})_{\Omega^*_0\setminus \Omega_0}+ (\nabla (v_0-\pi v_0), \nabla \tilde{\xi})_{ \Omega_0},\\
\mathcal{S}_6
=&(\nabla ( \mathcal{D}v_0  ), \nabla \tilde{\xi})_{\Omega\setminus\Omega^*_0}. 
\end{align*}

Applying Lemmas \ref{lem:h-N-4-1} and \ref{lem:superconvergence}  to  $\mathcal{S}_1$, then we  have
\begin{equation}\label{eq:diff-1}
\begin{aligned}
|( (w_1-\mathcal{A}w_1)_x,  \tilde{\xi}_x)_{  \Omega_{w_1} }|\le & C \sum_{i=0}^{N-1}\sum_{j=0}^{j^*} h_{j,y}^{k+1} \left\Vert  \frac{\partial^{k+2} w_1 }{\partial x \partial^{k+1} y}   \right\Vert_{\tau_{ij}} \Vert \tilde{\xi}_x \Vert_{\tau_{ij}}\\
\le & C \sum_{i=0}^{N-1}\sum_{j=0}^{j^*} h_{j,y}^{k+1} \varepsilon^{-(k+1)}e^{-\beta y_j/\varepsilon} h_{i,x}^{1/2}h_{j,y}^{1/2}  \Vert \tilde{\xi}_x \Vert_{\tau_{ij}}\\
\le &C  \varepsilon^{ 1/2} \sum_{i=0}^{N-1}\sum_{j=0}^{j^*} N^{-(k+2)} \Vert \tilde{\xi}_x \Vert_{\tau_{ij}}\\
\le & C \varepsilon^{ 1/2} N^{-(k+1)}  \Vert \tilde{\xi}_x \Vert 
\end{aligned}
\end{equation}
and similarly  obtain
\begin{equation}
\begin{aligned}
&|(\nabla (v_0-\mathcal{A} v_0), \nabla \tilde{\xi})_{\Omega\setminus\Omega^*_0}|+|( (w_1-\mathcal{A}w_1)_y,  \tilde{\xi}_y)_{ \Omega_{w_1} }|+|( (z_1-\mathcal{A}z_1)_x,  \tilde{\xi}_x)_{ \Omega_{z_1} }|
\le C\varepsilon^{-1/2} N^{-(k+1)}  \Vert \nabla \tilde{\xi}  \Vert.
\end{aligned}
\end{equation}

From Lemma \ref{lem:B-stability}, one has
\begin{equation}\label{eq:diff-2}
\begin{aligned}
|\mathcal{S}_2|\le &\Vert (\mathcal{B}_1w_1)_x \Vert_{ e(\Omega_{w_1}) } \Vert \nabla \tilde{\xi}  \Vert\\
\le &|e(\Omega_{w_1})|^{1/2}\Vert (\mathcal{B}_1w_1)_x \Vert_{\infty, e(\Omega_{w_1}) } \Vert \nabla \tilde{\xi}  \Vert\\
\le & C \varepsilon^{1/2} \Vert ( w_1)_x \Vert_{\infty, e(\Omega_{w_1}) } \Vert \nabla \tilde{\xi}  \Vert\\
\le & C \varepsilon^{1/2} N^{-\sigma} \Vert \nabla \tilde{\xi}  \Vert.
\end{aligned}
\end{equation}

Inverse inequalities yield
\begin{equation}\label{eq:diff-3}
\begin{aligned}
|\mathcal{S}_3| \le &(\Vert (\mathcal{B}_1w_1  )_y \Vert_{ e(\Omega_{w_1}) }+\Vert ( \mathcal{C}_1z_1 )_x \Vert_{e(\Omega_{z_1})}) \Vert \nabla \tilde{\xi}  \Vert\\
\le & C (h_{j^*,y}^{-1}\Vert  \mathcal{B}_1w_1   \Vert_{ e(\Omega_{w_1}) }+h_{i^*,x}^{-1}\Vert  \mathcal{C}_1z_1   \Vert_{e(\Omega_{z_1})}) \Vert \nabla \tilde{\xi}  \Vert\\
\le & C \varepsilon^{-1}( |e(\Omega_{w_1})|^{1/2} \Vert  \mathcal{B}_1w_1    \Vert_{ \infty,e(\Omega_{w_1}) }+ |e(\Omega_{z_1})|^{1/2}\Vert  \mathcal{C}_1z_1 \Vert_{\infty,e(\Omega_{z_1})}) \Vert \nabla \tilde{\xi}  \Vert\\
\le & C \varepsilon^{-1/2}(  \Vert   w_1   \Vert_{\infty, e(\Omega_{w_1}) }+  \Vert   z_1    \Vert_{\infty,e(\Omega_{z_1})}) \Vert \nabla \tilde{\xi}  \Vert\\
\le & C \varepsilon^{-1/2} N^{-\sigma} \Vert \nabla \tilde{\xi}  \Vert.
\end{aligned}
\end{equation}

Assumption \ref{assumption-regularity} and direct calculations yield
\begin{equation}\label{eq:diff-4}
\begin{aligned}
|\mathcal{S}_4|\le  (\Vert \nabla w_1 \Vert_{\Omega\setminus\Omega_{w_1}}+\Vert \nabla z_1 \Vert_{\Omega\setminus\Omega_{z_1}}) \Vert \nabla \tilde{\xi}  \Vert 
\le   C\varepsilon^{-1/2}N^{-\sigma} \Vert \nabla \tilde{\xi}  \Vert.
\end{aligned}
\end{equation}

 H\"{o}lder inequalities, inverse inequalities, Lemmas \ref{lem:h-N-4-1} and \ref{lem:main}  yield
\begin{equation}\label{eq:diff-5}
\begin{aligned}
|\mathcal{S}_5|\le & C    h^{-1}_{N/4-1} \Vert v_0-\pi v_0 \Vert_{\infty,\Omega^*_0\setminus \Omega_0 } \Vert \nabla \tilde{\xi} \Vert_{1,\Omega^*_0\setminus \Omega_0 }  
+C    N \Vert v_0-\pi v_0 \Vert_{ \infty, \Omega_0 } \Vert \nabla \tilde{\xi} \Vert_{1,\Omega_0}\\
\le & C   h^{-1}_{N/4-1}  N^{-(k+1)} h^{1/2}_{N/4-1}  \Vert \nabla \tilde{\xi} \Vert_{ \Omega^*_0\setminus \Omega_0 }  
+C    N  N^{-(k+1)}  \Vert \nabla \tilde{\xi} \Vert_{ \Omega_0 }\\
\le &
C \varepsilon^{-3/2}( N^{-(k+1)} + \varepsilon ^{1/2} N^{-k} )  \Vert  \tilde{\xi} \Vert_{\varepsilon}.
\end{aligned}
\end{equation}

From the definition of $\mathcal{D}v_0$, we have
\begin{equation}\label{eq:diff-6}
\begin{aligned}
|\mathcal{S}_6|\le &  C  
\Vert \nabla (\mathcal{D}v_0) \Vert_{\infty,\partial\Omega^*_0} \Vert  \nabla \tilde{\xi}  \Vert_{1,e(\Omega\setminus\Omega^*_0)}\\
\le &
C   {\color{blue}h^{-1}_{N/4-2} \Vert  \mathcal{A}v_0-\pi v_0  \Vert_{\infty,\partial\Omega^*_0} \cdot h^{1/2}_{N/4-2}\Vert  \nabla \tilde{\xi}  \Vert_{ e(\Omega\setminus\Omega^*_0) } }\\
\le &
C \varepsilon^{-3/2} \Vert  \mathcal{A}v_0-\pi v_0  \Vert_{\infty, \Omega^*_0} \Vert   \tilde{\xi}  \Vert_{\varepsilon}\\
\le &
C \varepsilon^{-3/2}  N^{-(k+1)}\Vert   \tilde{\xi}  \Vert_{\varepsilon} .
\end{aligned}
\end{equation}
where Lemmas \ref{lem:main} and \ref{lem:VEE-interpolation-error} yield
\begin{align*}
\Vert  \mathcal{A}v_0-\pi v_0  \Vert_{\infty, \Omega^*_0}
\le 
\Vert  \mathcal{A}v_0-  v_0  \Vert_{\infty, \Omega^*_0}
+\Vert   v_0-\pi v_0  \Vert_{\infty, \Omega^*_0}
\le C N^{-(k+1)}.
\end{align*}

Substituting \eqref{eq:diff-1}--\eqref{eq:diff-6} into
\eqref{eq:diffusion} and considering $\sigma\ge k+3/2$, we are done. 
\end{proof}

\begin{lemma}\label{lem:reaction}
Let Assumptions \ref{assumption-regularity} and \ref{assumption-varepsilon-N} hold. Let $\sigma\ge k+3/2$.  Then one has
$$
|\mathscr{S}_2|\le C\varepsilon^{1/2}N^{-(k+1)} \ln^{1/2} N \Vert \tilde{\xi} \Vert_{\varepsilon}.
$$
\end{lemma}
\begin{proof}
From \eqref{eq:idea-supercloseness}, we have 
\begin{equation}\label{eq:C1}
\mathscr{S}_2=( b (v_0-\mathcal{E}v_0),  \tilde{\xi})+\sum\limits_{i=1}^4( b (w_i-S_iw_i),  \tilde{\xi})+  \sum\limits_{i=1}^4(  b(z_i-T_iz_i),  \tilde{\xi})+a(\mathcal{C}(w),\tilde{\xi}).
\end{equation}

First, from \eqref{eq:projector-2-S}, H\"{o}lder inequalities and Lemma \ref{lem:h-N-4-1} we obtain
\begin{equation}\label{eq:C2}
\begin{aligned}
&|( b (v_0-\mathcal{E}v_0),  \tilde{\xi})|\le  |( b(v_0-\mathcal{A} v_0),   \tilde{\xi})_{\Omega\setminus\Omega^*_0}|+|(b(\mathcal{D}v_0 ),  \tilde{\xi})_{\Omega\setminus\Omega^*_0}|\\
&\le 
  C \Vert v_0-\mathcal{A}v_0 \Vert_{\infty,\Omega \setminus \Omega^*_0} \Vert \tilde{\xi} \Vert_{1,\Omega \setminus \Omega^*_0 }+C  \Vert  \mathcal{A}v_0-\pi v_0  \Vert_{\infty,\partial\Omega^*_0} \Vert   \tilde{\xi}  \Vert_{1, e(\Omega\setminus\Omega^*_0)}\\
&  \le 
  C \varepsilon^{1/2}  N^{-(k+1)} \ln^{1/2} N   \Vert  \tilde{\xi}
\Vert_{\Omega \setminus \Omega^*_0}
+
  C N^{-(k+1)} h_{N/4-2}^{1/2}  \Vert   \tilde{\xi}  \Vert_{ e(\Omega\setminus\Omega^*_0)  }.
\end{aligned}
\end{equation}

Second, H\"{o}lder inequalities and Lemma \ref{lem:h-N-4-1} yield
\begin{equation}\label{eq:C3}
\begin{aligned}
& a(\mathcal{C}(w),\tilde{\xi})\le \varepsilon^2 |(\nabla \mathcal{C}(w), \nabla \tilde{\xi} )|+ |(b\mathcal{C}(w), \tilde{\xi})|\\
\le & C \varepsilon^2  \Vert  \nabla \mathcal{C}(w) \Vert_{\infty,e(\Omega)} \Vert \nabla \tilde{\xi} \Vert_{1,e(\Omega) }+C\Vert   \mathcal{C}(w) \Vert_{\infty,e(\Omega)} \Vert\tilde{\xi}  \Vert_{1,e(\Omega)  }\\
\le & C \varepsilon^2  h^{-1}_0 N^{-\sigma}\cdot  h^{1/2}_0
 \Vert \nabla \tilde{\xi} \Vert_{ e(\Omega) }
 +C     N^{-\sigma}\cdot  h^{1/2}_0
 \Vert  \tilde{\xi} \Vert_{ e(\Omega) }\\
 \le & C\varepsilon^{1/2} N^{-(\sigma-1/2)}  \Vert   \tilde{\xi} \Vert_{\varepsilon}
 + C\varepsilon^{1/2} N^{-(\sigma+1/2)} \Vert   \tilde{\xi} \Vert_{\varepsilon},
\end{aligned}
\end{equation}
where we have  used   $\Vert   \mathcal{C}(w) \Vert_{\infty,e(\Omega)}\le C N^{-\sigma}$ and $C_6 \varepsilon N^{-1} \le h_0\le C_7 \varepsilon N^{-1}$.

 At last, we analyze $(b(w_1-S_1w_1),\tilde{\xi})$ and the remaining terms can be discussed in a similar way.  Lemmas  \ref{lem:VEE-interpolation-error} and \ref{lem:h-N-4-1} yield 
\begin{equation*}
\begin{aligned}
&\Vert w_1-\mathcal{A}w_1 \Vert_{\Omega_{w_1}}^2
=\sum_{i=0}^{N-1}\sum_{j=0}^{j^*} \Vert w_1-\mathcal{A}w_1 \Vert_{\tau_{i,j}}^2\\
\le & 
C  \sum_{i=0}^{N-1}\sum_{j=0}^{j^*} 
\left(\sum_{l+m=k+1} h_{i,x}^lh_{j,y}^m \left\Vert  \frac{ \partial^{k+1} w_1 }{\partial x^{l} \partial y^m } \right\Vert_{\tau_{i,j}} \right)^2\\
\le &
C \sum_{i=0}^{N-1}\sum_{j=0}^{j^*} 
\left(\sum_{l+m=k+1} h_{i,x}^lh_{j,y}^m \varepsilon^{-m}e^{-\beta y_j/\varepsilon}h_{i,x}^{1/2}h_{j,y}^{1/2}  \right)^2
\\
\le &
C\sum_{i=0}^{N-1}\sum_{j=0}^{j^*} 
\left(\sum_{l+m=k+1} N^{-(l+1/2)}\varepsilon^{m+1/2}N^{-(m+1/2)} \varepsilon^{-m}   \right)^2 \\
\le &
C \sum_{i=0}^{N-1}\sum_{j=0}^{j^*} (\varepsilon^{1/2}N^{-(k+2)})^2\le C \varepsilon N^{-2(k+1)}.
\end{aligned} 
\end{equation*}
  Lemma \ref{lem:B-stability} yields
\begin{align*}
\Vert  \mathcal{B}_1w_1  \Vert_{e(\Omega_{w_1})}
\le  
C |e(\Omega_{w_1})|^{1/2}\Vert w_1 \Vert_{\infty,e(\Omega_{w_1})}
\le 
Ch_{j^*}^{1/2} N^{-\sigma}.
\end{align*}
From the
Cauchy-Schwarz inequality one obtains
\begin{equation}\label{eq:C4}
\begin{aligned}
&(b(w_1-S_1w_1),\tilde{\xi})=(bw_1,\tilde{\xi})_{\Omega\setminus \Omega_{w_1} }+(b(w_1-\mathcal{A}w_1),\tilde{\xi})_{\Omega_{w_1}}
+(b\;\mathcal{B}_1w_1,\tilde{\xi})_{e(\Omega_{w_1})}\\
\le & C \Vert w_1 \Vert_{\Omega\setminus \Omega_{w_1} } \Vert \tilde{\xi} \Vert_{\Omega\setminus \Omega_{w_1} }
+ C \Vert w_1-\mathcal{A}w_1 \Vert_{\Omega_{w_1}} \Vert \tilde{\xi} \Vert_{\Omega_{w_1} }
+\Vert \mathcal{B}_1w_1 \Vert_{ e(\Omega_{w_1}) } \Vert \tilde{\xi} \Vert_{e(\Omega_{w_1})} \\
\le & C (\varepsilon^{1/2} N^{-\sigma}  
+  \varepsilon^{1/2} N^{-(k+1)}  +Ch_{j^*}^{1/2} N^{-\sigma}) \Vert \tilde{\xi} \Vert_{\varepsilon}\\
\le &C \varepsilon^{1/2} N^{-(k+1)}\Vert \tilde{\xi} \Vert_{\varepsilon}.
\end{aligned}
\end{equation}
Substituting \eqref{eq:C2}--\eqref{eq:C4} into \eqref{eq:C1}, we are done.
\end{proof}

Now we are in a position to present our supercloseness result in the balanced norm.
\begin{theorem}\label{the:supercloseness}
Let Assumptions \ref{assumption-regularity} and \ref{assumption-varepsilon-N} hold. Let $\sigma\ge k+3/2$. 
Let $P_su$ defined in \eqref{eq:projector-1-S} be the interpolation to the solution $u$ of \eqref{eq:reaction diffusion equation}. Let $u^N$ be the solution of \eqref{eq:weak formulation}.  Then one has
$$
 \Vert P_su-u^N \Vert_{b}\le   C ( \varepsilon^{1/2}N^{-k}+N^{-(k+1)}\ln^{1/2} N). 
$$

\end{theorem}
\begin{proof}
From \eqref{eq:idea-supercloseness}, Lemmas \ref{lem:diffusion} and \ref{lem:reaction}, we obtain
$$
\Vert \tilde{\xi} \Vert_{\varepsilon}\le
 C \varepsilon^{1/2}  (\varepsilon^{1/2}N^{-k}+N^{-(k+1)}\ln^{1/2} N),
$$
which implies the following estimations
\begin{align*}
&\varepsilon^{1/2} \Vert \nabla (P_su-u^N ) \Vert\le  C  ( \varepsilon^{1/2}N^{-k}+N^{-(k+1)}\ln^{1/2} N),\\
&\Vert P_su-u^N   \Vert\le  C \varepsilon^{1/2}  (\varepsilon^{1/2}N^{-k}+N^{-(k+1)}\ln^{1/2} N).
\end{align*}
Thus we are done.
\end{proof}
\begin{remark}
From Theorems \ref{the:main-result} and \ref{the:supercloseness}, we could conclude that 
Theorem \ref{the:supercloseness} presents a supercloseness result.   
It is the first supercloseness result  in the balanced norm in the literature. 

\end{remark}
 
%
%
%
\section{Numerical experiment}\label{sec:numerical}
In this section we present numerical experiments on Bakhvaolv-type rectangular meshes that support our theoretical results. All calculations were carried out by using Intel Visual Fortran 11 and  the discrete problems
were solved using GMRES (see, e.g., \cite{Ben1Gol2Lie3:2005-Numerical}).

We present errors and convergence orders in the computed solutions for the boundary value problem
\begin{equation}\label{eq:numerical-experiment}
\begin{aligned}
-\varepsilon^2\Delta u+2u&=f(x,y)\quad&&\text{in $\Omega=(0,1)^{2}$},\\
u&=0\quad&&\text{on $\partial\Omega$},
\end{aligned}
\end{equation}
where the right-hand side $f$ is chosen in such a way that
\begin{equation*}
u(x,y)=
\left(
1-\frac{e^{-x/\varepsilon}+e^{-(1-x)/\varepsilon}}{1+e^{-1/\varepsilon}}
\right)
\left(
1-\frac{e^{-y/\varepsilon}+e^{-(1-y)/\varepsilon}}{1+e^{-1/\varepsilon}}
\right)
\end{equation*}
is the exact solution.  For the computations we will assign values to the parameters  in \eqref{eq:Bakhvalov mesh-Roos} and  \eqref{eq:Bakhvalov mesh-Kopteva}.
 We set $\sigma=k+1$ when the error $\Vert u-u^N \Vert_b$ is discussed and $\sigma=k+3/2$ when the error $\Vert P_s u-u^N \Vert_b$  is discussed. Also, we set  $\beta=1$ and $C_1=4\sigma/(3\beta)$ in  \eqref{eq:Bakhvalov mesh-Kopteva}.

In our numerical tests, we will  consider $\varepsilon=10^{-3},10^{-4},\cdots,10^{-6}$, $k=1,2,3$ and $N=12,24,\cdots,768$. Our numerical experiments imply that 
meshes \eqref{eq:Bakhvalov mesh-Roos} and  \eqref{eq:Bakhvalov mesh-Kopteva} have same performances. Thus we only present numerical results for mesh \eqref{eq:Bakhvalov mesh-Roos}.

\par
For a fixed $\varepsilon$ and $N$, we have evaluated the error $e^{N,\varepsilon}_{c}=\Vert u-u^N \Vert_{b}$ and $e^{N,\varepsilon}_{s}=\Vert P_s u-u^N \Vert_{b}$,  where  $P_s u$ is the interpolation of the exact solution $u$  to \eqref{eq:numerical-experiment}, which is defined in \eqref{eq:projector-1-S} and $u^N$ represents its numerical approximation. In the following  we present the errors 
$$
e^N_c=\max\limits_{\varepsilon=10^{-3},\ldots,10^{-6}}e^{N,\varepsilon}_{c},\quad
e^N_s=\max\limits_{\varepsilon=10^{-3},\ldots,10^{-6}}e^{N,\varepsilon}_{s} 
$$
and the corresponding orders of convergence
$$
p^N_c=\frac{\ln e^N_c -\ln e^{2N}_c}{ \ln 2 },
\quad
p^N_s=\frac{\ln e^N_s -\ln e^{2N}_s}{ \ln 2 }.
$$



Numerical results are presented in Table \ref{table:alex-1} and log-log chart   \ref{fig:alex-1}.
Table \ref{table:alex-1} lists  errors in the balanced norm on Bakhvalov-type meshes \eqref{eq:Bakhvalov mesh-Roos} for $\varepsilon=10^{-3}, \ldots, 10^{-6}$ in the cases of $k=1$ and $k=2$.   Errors in the balanced norm  for the cases $k=3$ are plotted in Figure \ref{fig:alex-1}. The data in Table \ref{table:alex-1} and Figure \ref{fig:alex-1} show   optimal uniform convergence of $\Vert u-u^N \Vert_b$ with respect to the singular perturbation parameter $\varepsilon$, which implies that Theorem \ref{the:main-result} is sharp. The data also show that the superclosenss result is   of order $k+1$ for $k$th rectangular finite elements, which support Theorem \ref{the:supercloseness} in a sense
and also imply Theorem \ref{the:supercloseness} might not be sharp.

 {\color{red}
 \begin{figure}[htbp]
\centering
\includegraphics[width=5in]{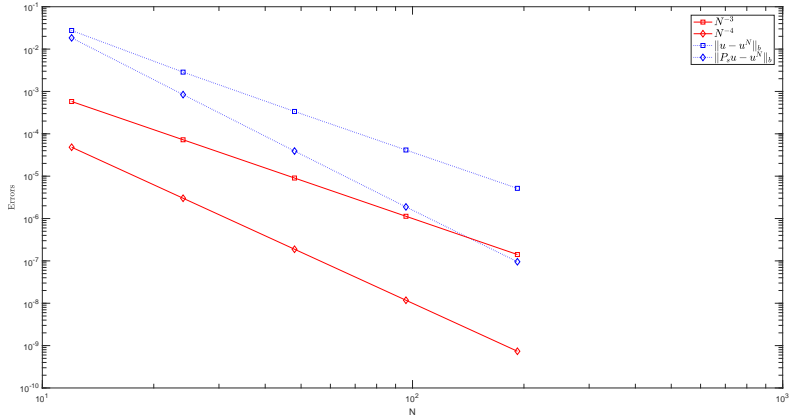}
\caption{Errors in the balanced norm when $k=3$.}
\label{fig:alex-1}
\end{figure}
}

\begin{table}[h]
\caption{Errors and rates  in the balanced norm for $k=1$ and $k=2$}
\footnotesize
\begin{tabular*}{\textwidth}{@{\extracolsep{\fill}} c| cc|cc|  cc|cc| }
\hline
     &$k=1$ & $k=1$  & $k=1$  &  $k=1$ &$k=2$ & $k=2$  & $k=2$  &  $k=2$     \\ 
\hline
 \diagbox{$N$}{ }   &$e^N_c$ &$p^N_c$  & $e^N_s$  &$p^N_s$  &$e^N_c$ &$p^N_c$  & $e^N_s$  &$p^N_s$    \\ 
\hline
             $12$       &  0.397E0  & 1.04     & 0.101E0   & 2.05   &  0.103E0  & 2.13     & 0.336E-1   &  3.43   \\
             $24$       &  0.193E0  & 1.00      & 0.245E-1   &  2.11 &  0.236E-1  & 2.03     & 0.312E-2   &  3.40  \\
             $48$       & 0.963E-1  & 1.00      & 0.566E-2   &  2.07   &  0.576E-2  & 2.01 & 0.295E-3   & 3.36  \\
             $96$       & 0.481E-1  & 1.00     &0.135E-2   & 2.04   &  0.143E-2  & 2.00     & 0.287E-4   & 3.29\\
             $192$       & 0.241E-1  & 1.01     & 0.328E-3   &  2.02 &  0.357E-3  & 2.00         & 0.295E-5   &  3.16 \\
             $384$       & 0.120E-1  & 1.00     & 0.808E-4   &  2.01  &  0.892E-4  & ---     & 0.328E-6   &  --- \\
             $768$       & 0.601E-2  & ---      & 0.201E-4   & ---   &  ---  & ---    & ---   & --- \\
\hline
\end{tabular*}
\label{table:alex-1}
\end{table}

  
%
%
\section*{References}
\bibliographystyle{plain}


%
%

\end{document}